\documentclass[reqno]{amsart}

\input xy
\xyoption{all}
\usepackage{epsfig}
\usepackage{color}
\usepackage{amsthm}
\usepackage{amssymb}
\usepackage{amsmath}
\usepackage{amscd}
\usepackage{amsopn}
\usepackage{url}
\usepackage{hyperref}\hypersetup{colorlinks}

\usepackage{color} 

\definecolor{darkred}{rgb}{1,0,0} 
\definecolor{darkgreen}{rgb}{0,0.8,0}
\definecolor{darkblue}{rgb}{0,0,1}

\hypersetup{colorlinks,
linkcolor=darkblue,
filecolor=darkgreen,
urlcolor=darkred,
citecolor=darkgreen}

\newcommand{\labell}[1] {\label{#1}}

\numberwithin{equation}{section}
\newtheorem {Theorem}{Theorem}
\numberwithin{Theorem}{section}

\newtheorem {Corollary}[Theorem]{Corollary}
\theoremstyle{definition}
\newtheorem{Definition}[Theorem]{Definition}
\theoremstyle{remark}
\newtheorem{Remark}[Theorem]{Remark}
\newtheorem{Example}[Theorem]{Example}

\expandafter\chardef\csname pre amssym.def at\endcsname=\the\catcode`\@
\catcode`\@=11
\def\undefine#1{\let#1\undefined}
\def\newsymbol#1#2#3#4#5{\let\next@\relax
 \ifnum#2=\@ne\let\next@\msafam@\else
 \ifnum#2=\tw@\let\next@\msbfam@\fi\fi
 \mathchardef#1="#3\next@#4#5}
\def\mathhexbox@#1#2#3{\relax
 \ifmmode\mathpalette{}{\m@th\mathchar"#1#2#3}%
 \else\leavevmode\hbox{$\m@th\mathchar"#1#2#3$}\fi}
\def\hexnumber@#1{\ifcase#1 0\or 1\or 2\or 3\or 4\or 5\or 6\or 7\or 8\or
 9\or A\or B\or C\or D\or E\or F\fi}

\font\teneufm=eufm10
\font\seveneufm=eufm7
\font\fiveeufm=eufm5
\newfam\eufmfam
\textfont\eufmfam=\teneufm
\scriptfont\eufmfam=\seveneufm
\scriptscriptfont\eufmfam=\fiveeufm

\catcode`\@=\csname pre amssym.def at\endcsname

\newcommand{\CA}{{\mathcal A}}

\newcommand{\CS}{{\mathcal S}}

\newcommand{\id}{{\mathit id}}
\newcommand{\pt}{{\mathit pt}}
\newcommand{\const}{{\mathit const}}

\newcommand{\tH}{\tilde{H}}

\newcommand{\A}{{\mathcal A}}
\newcommand{\tA}{\tilde{\mathcal A}}

\newcommand{\PP}{{\mathcal P}}
\newcommand{\bPP}{\bar{\mathcal P}}

\def    \F      {{\mathbb F}}

\def    \C      {{\mathbb C}}
\def    \R      {{\mathbb R}}

\def    \Z      {{\mathbb Z}}

\def    \Q      {{\mathbb Q}}

\def    \CP     {{\mathbb C}{\mathbb P}}

\def    \12    {{\frac{1}{2}}}

\def    \im     {\operatorname{im}}

\def    \HF     {\operatorname{HF}}

\def    \HQ     {\operatorname{HQ}}
\def    \GW     {\operatorname{GW}}
\def    \H     {\operatorname{H}}

\def    \CF     {\operatorname{CF}}

\def    \bx     {\bar{x}}
\def    \by     {\bar{y}}

\def    \bPhi     {\bar{\Phi}}

\def    \MUCZ  {\operatorname{\mu_{\scriptscriptstyle{CZ}}}}

\def    \s  {\operatorname{c}}

\def    \ssminus        {\smallsetminus}

\begin{document}


\setlength{\smallskipamount}{6pt}
\setlength{\medskipamount}{10pt}
\setlength{\bigskipamount}{16pt}





\title[Action--Index Relations]{Action--Index Relations for Perfect
  Hamiltonian Diffeomorphisms}

\author[Mike Chance]{Mike Chance}
\author[Viktor Ginzburg]{Viktor L. Ginzburg}
\author[Ba\c sak G\"urel]{Ba\c sak Z. G\"urel}

\address{MC and BG: Department of Mathematics, Vanderbilt University,
  Nashville, TN 37240, USA} 
\email{michael.j.chance@vanderbilt.edu}
\email{basak.gurel@vanderbilt.edu}

\address{VG: Department of Mathematics, UC Santa Cruz, Santa Cruz, CA
  95064, USA} \email{ginzburg@ucsc.edu}


\date{\today} 

\thanks{The work is partially supported by NSF and by the faculty
  research funds of the University of California, Santa Cruz.}


\begin{abstract}
  We show that the actions and indexes of fixed points of a
  Hamiltonian diffeomorphism with finitely many periodic points must
  satisfy certain relations, provided that the quantum cohomology of
  the ambient manifold meets an algebraic requirement satisfied for
  projective spaces, Grassmannians and many other manifolds. We also
  refine a previous result on the Conley conjecture for negative
  monotone symplectic manifolds, due to the second and third authors,
  and show that a Hamiltonian diffeomorphism of such a manifold must
  have simple periodic orbits of arbitrarily large period whenever its
  fixed points are isolated.

\end{abstract}

\maketitle

\tableofcontents

\section{Introduction and main results}
\labell{sec:main-results}

\subsection{Introduction}
\label{sec:intro}
The central theme of this paper is a rigidity feature of Hamiltonian
diffeomorphisms with finitely many periodic points. To be more
specific, we prove that there are certain relations between the
actions and indexes of fixed points of such a diffeomorphism, provided
that the quantum cohomology of the ambient manifold meets an
algebraic requirement, which is fully satisfied for projective spaces,
Grassmannians, and to a degree in all known examples. We also refine
our previous result on the Conley conjecture for negative monotone
symplectic manifolds (see \cite {GG:neg-mon}) and show that a
Hamiltonian diffeomorphism of such a manifold must have simple
periodic orbits of arbitrarily large period whenever its fixed points
are isolated.

The fact that there are some relations between the mean indexes and/or
actions of fixed points of a Hamiltonian system with finitely many
periodic orbits is not new. For indexes of a Reeb flow on the standard
sphere it goes back to \cite{E,EH,Vi} and it has been further explored
and generalized since then; see, e.g., \cite{Es,GK}. On the other
hand, the observation that the mean indexes, or indexes and actions,
of fixed points of a Hamiltonian diffeomorphism (of, say, $\CP^n$)
with finitely many periodic orbits must meet certain algebraic conditions is
relatively recent; see \cite{GK} for resonance relations for indexes
and \cite{GG:gaps} for action--index relations.  (Interestingly, no
analogue of action--index relations in the contact case appears to
have been established so far.) The main focus of this paper is a
treatment of the action--index relations in a more systematic way and
connecting it with the algebraic properties of the quantum homology of
the ambient manifold.

The paper is organized as follows. Below, in Sections
\ref{sec:relations} and \ref{sec:conley-thm}, we state the main results
of the paper --- the action--index relations and a refinement of the
Conley conjecture for negative monotone manifolds. In Section
\ref{sec:prelim}, we set our conventions and notation and discuss some
standard (and not entirely standard) notions and results from
symplectic topology, needed for the proof of the main theorems. These
include the mean index, the filtered and local Floer homology, the
action selector and the action selector carrier, and their
properties. Then, in Section \ref{sec:proof-relations}, we prove
Theorem \ref{thm:relations} establishing the existence of
action--index relations. Theorem \ref{thm:conley} (a refinement of the
Conley conjecture for negative monotone manifolds) is proved in
Section \ref{sec:proof-conley}.

\subsection{Action--index relations} 
\label{sec:relations}
Throughout this section and the rest of the paper, $(M,\omega)$ will
stand for a monotone or negative monotone closed symplectic manifold
of dimension $2n$ with monotonicity constant $\lambda$ and minimal
Chern number $N$. (We refer the reader to Section \ref{sec:prelim} for
the definitions and a detailed discussion of the notions used in this section.)
Here we only note that we may assume from now on that
$0<N<\infty$, i.e., $\left<c_1(TM),\pi_2(M)\right>\neq 0$ in $\R$. For
otherwise the conditions of Theorem \ref{thm:relations} below are
never satisfied and Theorem \ref{thm:conley} is known to hold. (This
is a consequence of the Conley conjecture for symplectic manifolds
with $c_1(TM)|_{\pi_2(M)}= 0$, proved in \cite {GG:gaps,He:irr}; see
also \cite{Gi:conley,Hi}.)

Let $\varphi_H$ be a Hamiltonian diffeomorphism
of $M$ generated by a one-periodic in time Hamiltonian $H\colon S^1\times M\to \R$ with
$S^1=\R/\Z$.  We denote by $\PP$ the collection of contractible
one-periodic orbits of the time-dependent flow $\varphi_H^t$ and by
$\bPP$ the collection of capped one-periodic orbits. Clearly, $\PP$ can
be viewed as a subset of the fixed point set of $\varphi_H$. For $x\in\PP$,
the \emph{augmented action} of $H$ on $x$ is defined by
$$
\tA_H(x)=\A_H(\bx)-\frac{\lambda}{2} \Delta_H(\bx),
$$
where $\A_H(\bx)$ and $\Delta_H(\bx)$ stand for the action and,
respectively, the mean index of a capped orbit $\bx$; see Section
\ref{sec:conventions}. Since $\A_H(\bx)$ and $\Delta_ H(\bx)$ change
in the same way under recapping, up to the factor $\lambda/2$, the
augmented action is well defined, i.e., independent of the capping.

We denote the quantum cohomology of $M$ over a ground field $\F$ by
$\HQ^*(M;\F)$ or simply by $\HQ^*(M)$ when the role of $\F$ is
inessential. The quantum product of elements $u$ and $v$ in $\HQ^*(M)$
will be denoted by $u*v$ and the degree of $u$ by $|u|$. Recall that
the quantum cohomology is a module over the Novikov ring $\Lambda$. The
version of $\Lambda$ we will utilize here is a completion of the
polynomial ring $\F[q,q^{-1}]$, where $|q|=2N$. We refer the reader to
Section \ref{sec:QH} for more details on quantum (co)homology and
further references.

Below we will use the  ceiling and floor functions $\lceil a \rceil$
and $\lfloor a \rfloor$. Recall that $\lceil a \rceil$ stands for the
least integer greater than or equal to $a$ and $\lfloor a \rfloor$ is the
greatest integer smaller than or equal to $a$.

The main result of the paper is the following.

\begin{Theorem}[Action--index relations]
\label{thm:relations}
Let $\varphi_H$ be a Hamiltonian diffeomorphism with finitely many
periodic orbits of a closed monotone symplectic manifold $M$ with
$0<N<\infty$.
\begin{itemize}
\item[(i)] Assume that 
\begin{equation}
\label{eq:product}
u_0*u_1*\cdots *u_{\ell} =q^\nu u_0 \text{ in } \HQ^*(M), 
\end{equation}
where $\nu>0$, the classes $u_1,\ldots,u_{\ell}$
have positive degree, and 
\begin{equation}
\label{eq:deg}
|u_1|+\cdots+|u_{\ell-1}|< 2N.
\end{equation}
Assume furthermore, unless $\nu=1$, that $\varphi_H$ is
non-degenerate. Then there exist $\ell$ distinct contractible
one-periodic orbits $x_0,\ldots,x_{\ell-1}$ of $\varphi_H^t$ such that
\begin{equation}
\label{eq:action-index1}
\tA_H(x_0)=\cdots=\tA_H(x_{\ell-1}).
\end{equation}
\item[(ii)] Alternatively, assume that there exists $u\in
  \HQ^{*>0}(M)$ such that 
\begin{equation}
\label{eq:case2}
u^d\neq 0
\end{equation} 
for some sufficiently large $d$ (e.g., we can take $d=\big\lceil
2N|\PP|/|u|\big\rceil +1 $) and that $\varphi_H$ is
non-degenerate. Then the assertion of (i) holds for $\ell=\big\lfloor
2N/|u|\big\rfloor$.
\end{itemize}

\end{Theorem}

We will call \eqref{eq:product} a \emph{product decomposition} of
$u_0$. It is essential that the degree bound \eqref{eq:deg} involves
neither the first term $u_0$ of the product decomposition nor the last
one $u_\ell$. In particular, both $u_0$ and $u_\ell$ can have large
degree not necessarily bounded by $2N$. (However, clearly,
$|u_1|+\ldots+|u_\ell|=2N\nu$ and \eqref{eq:deg} is automatically
satisfied when $\nu=1$.) It is also worth pointing out that in all
known examples, which are in fact listed below, of Hamiltonian
diffeomorphisms $\varphi_H$ with finitely many periodic points,
$\varphi_H$ is non-degenerate and every periodic point of $\varphi_H$
is a fixed point, i.e., $\varphi_H$ is \emph{perfect} in the
terminology of \cite{GK} -- hence the title of the paper.

The theorem cannot produce more than $\ell=2N$ fixed points with equal
augmented action. (In fact, $\ell\leq N$ since elements of odd degree
appear to never contribute to product decompositions.)  A difficulty
here lies in showing that the points are distinct and is similar to
the difficulty arising in establishing the degenerate case of the
Arnold conjecture. However, as we will see, in all known examples all
fixed points have equal augmented actions.

Let us now examine examples of Hamiltonian diffeomorphisms with
finitely many periodic points.

\begin{Example} 
\label{ex:CPn}
The simplest Hamiltonian diffeomorphism with finitely many periodic
orbits is an irrational rotation of $S^2$. More generally, let
$\varphi_H$ be the Hamiltonian diffeomorphism of $\CP^n$ generated by
a quadratic Hamiltonian
$H(z)=\pi\big(\lambda_0|z_0|^2+\cdots+\lambda_n|z_n|^2\big)$, where
the coefficients $\lambda_0,\ldots,\lambda_n$ are all distinct. (Here,
we have identified $\CP^n$ with the quotient of the unit sphere in
$\C^{n+1}$.) Now, \eqref{eq:product} takes the form $u^{n+1}=q$, where
$u$ is the first Chern class of the canonical line bundle, and
$N=n+1=\ell$; see, e.g., \cite[Section 11.3]{MS}.  The Hamiltonian
diffeomorphism $\varphi_H$ is perfect and has exactly $\ell$
fixed points (the coordinate axes). Their augmented actions are equal
to $\pi\sum \lambda_i/(n+1)$. In this connection let us point out that
there is a sign error in \cite[Example 1.2]{GK}. With our conventions
(see Section \ref{sec:conventions}),
$\Delta(x_j)=-\sum\lambda_i+(n+1)\lambda_j$, where $x_j$ is equipped
with the trivial capping.
\end{Example}

This example fits in a much more general framework:

\begin{Example} 
\label{ex:Ham-gp-action}
Suppose that $(M, \omega)$ admits a Hamiltonian action of a torus $G$ with
isolated fixed points; see, e.g., \cite{GGK} for the definition and
further details. A generic element of $G$ gives rise to a
perfect Hamiltonian diffeomorphism $\varphi$ of $(M, \omega)$ whose
fixed points are exactly the fixed points of the torus action.  One
can show that in this case all fixed points have the same augmented
action, i.e., \eqref{eq:action-index1} holds for the entire collection
of fixed points.  

One way to see that this is true is as follows. Without loss of
generality, we may assume that $\lambda/2=1$. Let $\vec{H}$ be the
moment map of the action. Then the equivariant two-form
$\omega^G=\omega+\vec{H}$ can be thought as simultaneously representing the
equivariant Chern class $c_1^G(TM)$ and an equivariant
extension $[\omega]^G$ of the symplectic form class.  Since the fixed
points of the action are isolated, the localization of the latter
class is simply the sum of the moment map values $\vec{H}(x)$ for $x\in M^G$. On the
other hand, it is not hard to see that for $c^G_1(TM)$ and each $x$
this is also the sum of the weights of the representation of $G$ on $T_x
M$. Expressing $H(x)$ as a linear combination of the components of $\vec{H}(x)$ and
$\Delta(x)$ as a linear combination of the components of the weights, we conclude that
$H(x)=\Delta(x)$. (The authors are grateful to Yael Karshon for this
argument.)

Examples of symplectic manifolds which admit such torus actions
include a majority of coadjoint orbits of compact Lie groups, e.g.,
complex Grassmannians $G(k,N)$. One can also construct new examples
from a given one by equivariantly blowing-up the symplectic manifold
at fixed points. The resulting symplectic manifold always inherits
a Hamiltonian torus action and, in many instances, this action also
has isolated fixed points.
\end{Example}

\begin{Example}
\label{ex:anosov-katok}
Hamiltonian diffeomorphisms with finitely many periodic orbits need
not be associated with Hamiltonian torus actions as in Example
\ref{ex:Ham-gp-action}. For instance, there exists a Hamiltonian
perturbation $\varphi$ of an irrational rotation of $S^2$ with exactly
three ergodic invariant measures: the Lebesgue measure and the two
measures corresponding to the fixed points of $\varphi$;
\cite{AK,FK}. Taking direct products of such diffeomorphisms we obtain
examples in higher dimensions. It is easy to see from the construction
of $\varphi$ that in this case all fixed points have again the same
augmented action. (This also follows from Theorem \ref{thm:relations}
since $\varphi$ has exactly two fixed points.)
\end{Example}

To the best of our knowledge, Examples \ref{ex:Ham-gp-action} and
\ref{ex:anosov-katok} and their products exhaust the list of known
Hamiltonian diffeomorphisms with finitely many periodic orbits.

Product decompositions with $\ell\geq 2$ depending on the manifold in
question exist for Grassmannians $G(k,N)$ and their monotone products,
as is easy to see from the description of the quantum product on $G(k,N)$; see,
e.g., \cite{ST} and \cite{MS,Ta} for further references and
details. Moreover, once a product decomposition exists for $M$ it also
exists, with the same $\ell$, for any monotone product of the form
$M\times P$. (This follows from the quantum K\"unneth formula;
\cite{Ka}.)  Here, for instance, $P$ can be symplectically aspherical
although in this case no products $M\times P$
admitting perfect Hamiltonian diffeomorphisms are known. However, ``long''
product decompositions are difficult find. For instance, we have not
been able to show that in general $G(k,N)$ admits a product
decomposition of length $N$ (the minimal Chern number) and this is
where Case (ii) of the theorem becomes useful.

\begin{Corollary}
\label{cor:grassmannians}
Let $M$ be $G(k,N)$ or, more generally, a monotone product
$G(k,N)\times P$, and let $\varphi_H$ be a Hamiltonian diffeomorphism
of $M$ with finitely many fixed points.  Then there exist $\ell=N$
distinct fixed points $x_0,\ldots,x_{\ell-1}$ of $\varphi_H$
satisfying~\eqref{eq:action-index1}.
\end{Corollary}

\begin{proof} Let $u$ be the first Chern class of the canonical vector
  bundle over $G(k,N)$. Then $u^d\neq 0$, for any $d>0$, in
  $\HQ^*(G(k,N);\Q)$.  This is an immediate consequence of quantum
  Schubert calculus and, more precisely, of the quantum Pieri formula;
  see \cite{Be} and also, e.g., \cite{Ta}. Applying the quantum
  K\"unneth formula, we see that $u^d\neq 0$ for $M=G(k,N)\times P$;
  see \cite{Ka} and also \cite[Exercise 11.1.15]{MS}. Now the
  corollary follows from Case (ii) of Theorem \ref{thm:relations}.
\end{proof}

\begin{Remark} 
  The hypotheses of Cases (i) and (ii) of the theorem are in fact
  related. For instance, when the ground field $\F$ is finite, the
  existence of an element $u$ with $u^d\neq 0$ for all $d\geq 0$
  implies, by the pigeonhole principle, a product decomposition of
  length $\ell=\big[2N/|u|\big]$. This argument however cannot be used
  to find a product decomposition for $G(k,N)$ of length $N$: $u^3=0$
  in $\HQ^*(G(2,4);\Z_2)$ and it is absolutely essential for the proof
  of Corollary \ref{cor:grassmannians} that $\F$ has zero
  characteristic.
\end{Remark}

\begin{Remark}
  Theorem \ref{thm:relations} generalizes Corollary 1.11 and Theorem
  1.12 in \cite{GG:gaps}, where the augmented action was originally
  defined. However, a similar notion (the action--index) was
  considered in \cite{Po97} in a different context for a loop of
  Hamiltonian diffeomorphisms. Furthermore, replacing the mean index
  in the definition of the augmented action by some other version of
  the Conley--Zehnder index (as in, e.g., \cite[Section 1.6]{Sa}), we
  still obtain a well-defined, i.e., independent of capping,
  invariant. Theorem \ref{thm:relations} would not hold for such an
  invariant. What sets apart the augmented action, defined as above
  using the mean index, is that it is homogeneous with respect to
  iterations of $\varphi_H$.
\end{Remark}

\subsection{Conley conjecture for negative monotone symplectic
  manifolds} 
\label{sec:conley-thm}
Our proof of Theorem \ref{thm:relations} builds on the proof of the
Conley conjecture for negative monotone symplectic manifolds in
\cite{GG:neg-mon}. In turn, the proof of Theorem \ref{thm:relations}
lends itself readily to the following refinement of the main result of
\cite{GG:neg-mon}.

\begin{Theorem}
\label{thm:conley}
Let $\varphi_H$ be a Hamiltonian diffeomorphism of a closed, negative
monotone symplectic manifold. Assume that $\varphi_H$ has finitely
many fixed points. Then $\varphi_H$ has simple periodic orbits of
arbitrarily large period.
\end{Theorem}

Negative monotone manifolds exist in abundance. Among them are,
for instance, the hypersurfaces $z_0^k+\cdots+z_n^k=0$ in $\CP^n$ with
$N=k-(n+1)>0$; see, e.g., \cite[pp.\ 429--430]{MS}.

\begin{Remark}
  The new point here, as compared to \cite{GG:neg-mon}, is the
  existence of simple periodic orbits with arbitrarily large period
  rather than just the existence of infinitely many periodic orbits.

  The proof of Theorem \ref{thm:conley} utilizes
  Hamiltonian Floer theory. Hence, unless $M$ is required to be weakly
  monotone, the argument ultimately, although not explicitly, relies
  on the machinery of multi-valued perturbations and virtual cycles;
  see Remark \ref{rmk:Floer} for further discussion.
\end{Remark}

\subsection{Acknowledgements} The authors are grateful to Yael Karshon
and Yaron Ostrover for useful discussions.

\section{Preliminaries}
\label{sec:prelim}

The goal of this section is to set notation and conventions, following
mainly \cite{GG:gaps,GG:neg-mon}, and to give a brief review of Floer
homology and several other notions used in the paper.

\subsection{Conventions and notation}
\label{sec:conventions}
Let $(M^{2n},\omega)$ be a closed symplectic manifold. Throughout the
paper, we will usually assume for the sake of simplicity that $M$ is
\emph{rational}, i.e., the group $\left<[\omega],
  {\pi_2(M)}\right>\subset\R$ formed by the integrals of $\omega$ over
the spheres in $M$ is discreet. This condition is obviously satisfied
when $M$ is \emph{monotone} as in Theorem \ref{thm:relations} or
\emph{negative monotone} as in Theorem \ref{thm:conley}, i.e.,
$[\omega]\mid_{\pi_2(M)}=\lambda c_1(M)\!\mid_{\pi_2(M)}$ for some
$\lambda\geq 0$ in the former case or $\lambda<0$ in the
latter. Recall also that $M$ is called \emph{symplectically
  aspherical} if $[\omega]\mid_{\pi_2(M)}=0=c_1(M)\!\mid_{\pi_2(M)}$.

All Hamiltonians $H$ on $M$ considered in this paper are assumed to be
$k$-periodic in time, i.e., $H\colon S^1_k\times M\to\R$, where
$S^1_k=\R/k\Z$, and the period $k$ is always a positive integer.  When
the period is not specified, it is equal to one, which is the default
period in this paper. We set $H_t = H(t,\cdot)$ for $t\in
S^1=\R/\Z$. The Hamiltonian vector field $X_H$ of $H$ is defined by
$i_{X_H}\omega=-dH$. The (time-dependent) flow of $X_H$ will be
denoted by $\varphi_H^t$ and its time-one map by $\varphi_H$. Such
time-one maps are referred to as \emph{Hamiltonian diffeomorphisms}.
A one-periodic Hamiltonian $H$ can always be treated as
$k$-periodic. In this case, we will use the notation $H^{\sharp k}$
and, abusing terminology, call $H^{\sharp k}$ the $k$th iteration of
$H$.

Let $K$ and $H$ be one-periodic Hamiltonians such that $K_1=H_0$ and
$H_1=K_0$. We denote by $K\sharp H$ the two-periodic Hamiltonian equal
to $K_t$ for $t\in [0,\,1]$ and $H_{t-1}$ for $t\in [1,\,2]$. Thus,
$H^{\sharp k}=H\sharp \cdots\sharp H$ ($k$ times).

Let $x\colon S^1_k\to W$ be a contractible loop. A \emph{capping} of
$x$ is a map $u\colon D^2\to M$ such that $u\mid_{S^1_k}=x$. Two
cappings $u$ and $v$ of $x$ are considered to be equivalent if the
integrals of $\omega$ and $c_1(TM)$ over the sphere obtained by
attaching $u$ to $v$ are equal to zero. A capped closed curve
$\bar{x}$ is, by definition, a closed curve $x$ equipped with an
equivalence class of cappings. In what follows, the presence of
capping is always indicated by the bar.

The action of a one-periodic Hamiltonian $H$ on a capped closed curve
$\bar{x}=(x,u)$ is defined by
$$
\CA_H(\bar{x})=-\int_u\omega+\int_{S^1} H_t(x(t))\,dt.
$$
The space of capped closed curves is a covering space of the space of
contractible loops and the critical points of $\CA_H$ on the covering
space are exactly capped one-periodic orbits of $X_H$. The
\emph{action spectrum} $\CS(H)$ of $H$ is the set of critical values
of $\CA_H$. This is a zero measure set; see, e.g., \cite{HZ,Sch}. When
$M$ is rational, $\CS(H)$ is closed, and hence nowhere dense.
Otherwise, $\CS(H)$ is dense in $\R$. These definitions extend to
$k$-periodic orbits and Hamiltonians in an obvious way. Clearly, the
action functional is homogeneous with respect to iteration:
$$
\CA_{H^{\sharp k}}(\bx^k)=k\CA_H(\bx).
$$
Here $\bx^k$ stands for the $k$th iteration of the capped orbit $\bx$.

The results of this paper concern only contractible periodic orbits
and throughout the paper \emph{a periodic orbit is always assumed
  to be contractible, even if this is not explicitly stated}.

A periodic orbit $x$ of $H$ is said to be \emph{non-degenerate} if the
linearized return map $d\varphi_H \colon T_{x(0)}W\to T_{x(0)}W$ has
no eigenvalues equal to one. Following \cite{SZ}, we call $x$
\emph{weakly non-degenerate} if at least one of the eigenvalues is
different from one. A Hamiltonian is non-degenerate if all its
one-periodic orbits are non-degenerate.

Let $\bar{x}$ be a non-degenerate (capped) periodic orbit.  The
\emph{Conley--Zehnder index} $\MUCZ(\bar{x})\in\Z$ is defined, up to a
sign, as in \cite{Sa,SZ}. (Sometimes, we will also use the notation
$\MUCZ(H,\bx)$.) More specifically, in this paper, the Conley--Zehnder
index is the negative of that in \cite{Sa}. In other words, we
normalize $\MUCZ$ so that $\MUCZ(\bar{x})=n$ when $x$ is a
non-degenerate maximum (with trivial capping) of an autonomous
Hamiltonian with small Hessian. The \emph{mean index}
$\Delta_H(\bx)\in\R$ measures, roughly speaking, the total angle swept
by certain eigenvalues with absolute value one of the linearized flow
$d\varphi^t_H$ along $x$ with respect to the trivialization associated
with the capping; see \cite{Lo,SZ}. The mean index is defined
regardless of whether $x$ is degenerate or not and $\Delta_H(\bx)$
depends continuously on $H$ and $\bx$ in the obvious sense. When $x$
is non-degenerate or just weakly non-degenerate, we have
$$
0<|\Delta_H(\bx)-\MUCZ(H,\bx)|<n.
$$
Furthermore, the mean index is homogeneous with respect to iteration:
$$
\Delta_{H^{\sharp k}}(\bx^k)=k\Delta_H(\bx).
$$
As a consequence, the augmented action is also homogeneous.

\subsection{Floer and quantum (co)homology} 
Although the hypotheses of Theorem \ref{thm:relations} are stated in
terms of quantum cohomology, we find working with homology more
intuitive in the context of Ljusternik--Schnirelman theory, which the proof
of the theorem relies on. Hence, here we focus mainly on the
definitions of quantum and Floer homology and just briefly indicate
the modifications needed in the case of cohomology. The assumptions of
Theorem \ref{thm:relations} are reformulated via homology at the
beginning of Section \ref{sec:proof-relations}.

\subsubsection{Floer homology}
In this subsection, we very briefly recall, mainly to set notation,
the construction of the filtered Floer homology. We refer the reader
to, e.g., \cite{HS,MS,Sa,SZ} and also \cite{FO,LT} for detailed
accounts and additional references.

Fix a ground field $\F$. Let $H$ be a non-degenerate Hamiltonian on
$M$.  Denote by $\CF^{(-\infty,\, b)}_k(H)$, where $b\in
(-\infty,\,\infty]$ is not in $\CS(H)$, the vector space of formal
sums
$$ 
\sigma=\sum_{\bar{x}\in \bPP(H)} \sigma_{\bar{x}}\bar{x}. 
$$
Here $\sigma_{\bar{x}}\in\F$ and $\MUCZ(\bar{x})+n=k$ and
$\CA_H(\bar{x})<b$. Furthermore, we require, for every $a\in \R$, the
number of terms in this sum with $\sigma_{\bar{x}}\neq 0$ and
$\CA_H(\bar{x})>a$ to be finite. We say that $\bx$ \emph{enters} the
chain $\sigma$ when $\sigma_{\bx}\neq 0$. The graded $\F$-vector space
$\CF^{(-\infty,\, b)}_*(H)$ is endowed with the Floer differential
counting the anti-gradient trajectories of the action functional; see,
e.g., \cite{HS,MS,Ono:AC,Sa} and also \cite{FO,LT}. Thus, we obtain a
filtration of the total Floer complex $\CF_*(H):=\CF^{(-\infty,\,
  \infty)}_*(H)$. Furthermore, we set $\CF^{(a,\,
  b)}_*(H):=\CF^{(-\infty,\, b)}_*(H)/\CF^{(-\infty,\,a)}_*(H)$, where
$-\infty\leq a<b\leq\infty$ are not in $\CS(H)$. The resulting
homology, the \emph{filtered Floer homology} of $H$, is denoted by
$\HF^{(a,\, b)}_*(H)$ and by $\HF_*(H)$ when
$(a,\,b)=(-\infty,\,\infty)$. Note that every $\F$-vector space
$\CF_k(H)$ is finite-dimensional when $M$ is negative monotone or
monotone with $\lambda>0$ and $0<N<\infty$.

The total Floer complex and homology are modules over the
\emph{Novikov ring} $\Lambda$. In this paper, the latter is defined as
follows. Let $\omega(A)$ and $\left<c_1(TM),A\right>$ denote the
integrals of $\omega$ and, respectively, $c_1(TM)$ over a cycle $A$.
Set
$$
I_\omega(A)=-\omega(A)\text{ and } I_{c_1}(A)=-2\left<c_1(TM),
  A\right>,
$$
where $A\in\pi_2(M)$.  For instance,
$$
I_\omega=\frac{\lambda}{2}I_{c_1}
$$
when $M$ is monotone or negative monotone, and, in particular,
\emph{$I_\omega(A)$ and $I_{c_1}(A)$ have opposite signs when $M$ is
  negative monotone}.  Let
$$
\Gamma=\frac{\pi_2(M)}{\ker I_\omega\cap \ker I_{c_1}}.
$$
Thus, $\Gamma$ is the quotient of $\pi_2(M)$ by the equivalence
relation where two spheres $A$ and $A'$ are considered to be
equivalent if $\omega(A)=\omega(A')$ and $\left<c_1(TM),
  A\right>=\left<c_1(TM), A'\right>$. 
The homomorphisms $I_\omega$ and $I_{c_1}$ descend to $\Gamma$ from
$\pi_2(M)$.

The group $\Gamma$ acts on $\CF_*(H)$ and on $\HF_*(H)$ via recapping:
an element $A\in \Gamma$ acts on a capped one-periodic orbit $\bar{x}$
of $H$ by attaching the sphere $A$ to the original capping. We denote
the resulting capped orbit by $\bx\# A$. Then,
$$
\MUCZ(\bx\# A)=\MUCZ(\bx)+ I_{c_1}(A)
\text{ and }
\CA_H(\bx\# A)=\CA_H(\bx)+I_\omega(A).
$$
In a similar vein, we also have
$$
\Delta_H(\bx\# A)=\Delta_H(\bx)+ I_{c_1}(A),
$$
regardless of whether $x$ is non-degenerate or not.

The Novikov ring $\Lambda$ is a certain completion of the group ring
$\F[\Gamma]$ of $\Gamma$ over $\F$. Namely, $\Lambda$ comprises formal
linear combinations $\sum \alpha_A e^A$, where $\alpha_A\in\F$ and
$A\in \Gamma$, such that for every $a\in \R$ the sum contains only
finitely many terms with $I_\omega(A) > a$ and
$\alpha_A\neq 0$. The Novikov ring $\Lambda$ is graded by setting
$|e^A|=I_{c_1}(A)$ for $A\in\Gamma$.  The action of $\Gamma$ turns
$\CF_*(H)$ and $\HF_*(H)$ into $\Lambda$-modules.

\begin{Example}
\label{ex:Novikov}
When $M$ is monotone or negative monotone with $\lambda\neq 0$ and
$0<N<\infty$, i.e., $\left<c_1(TM),\pi_2(M)\right>\neq 0$, we have
$\Gamma=\pi_2(M)/\ker I_{c_1}\simeq \Z$. Denote by $A$ the
generator of $\Gamma$ with $I_{c_1}(A)=-2N$ and set $q=e^A\in
\Lambda$.  Then $|q|=-2N$ and the Novikov ring $\Lambda$ is the ring
of Laurent series $\F[q^{-1},q]]$ in the monotone case and the ring
$\F[q,q^{-1}]]$ of Laurent series in $q^{-1}$ when $M$ is negative
monotone. Note that $\Lambda$, essentially by definition, is a
completion of the polynomial ring $\F[\Gamma]=\F[q,q^{-1}]$ with
respect to the valuation $I_\omega$.
\end{Example}

The definition of Floer homology extends to all, not necessarily
non-degenerate, Hamiltonians by continuity.  Let $H$ be an arbitrary
(one-periodic in time) Hamiltonian on $M$ and let the end points $a$
and $b$ of the action interval be outside $\CS(H)$. We
set
$$
\HF^{(a,\, b)}_*(H)=\HF^{(a,\, b)}_*(\tH),
$$
where $\tH$ is a non-degenerate, small perturbation of $H$. It is well
known that the right hand side is independent of $\tH$ as long as
the latter is sufficiently close to $H$.  Working with filtered Floer
homology, \emph{we will always assume that the end points of the
  action interval are not in the action spectrum.} (At this point the
background assumption that $M$ is rational becomes essential; see
\cite{He:irr} for the irrational case and also \cite[Remark
2.3]{GG:gaps}.)

The total Floer homology is independent of the Hamiltonian and 
isomorphic to the homology of $M$. More precisely, we have
$$
\HF_*(H)\cong \H_ {*}(M;\F)\otimes \Lambda
$$
as graded $\Lambda$-modules.

\begin{Remark}
\label{rmk:Floer}
We conclude this discussion by recalling that in order for the Floer
differential to be defined certain regularity conditions must be
satisfied generically. To ensure this, we have to either require $M$
to be weakly monotone (see \cite{HS,MS,Ono:AC,Sa}) or utilize the
machinery of virtual cycles (see \cite{FO,FOOO,LT} or, for the
polyfold approach, \cite{HWZ,HWZ2} and references therein). In the
latter case the ground field $\F$ is required to have zero
characteristic.  Here we are primarily interested in monotone
manifolds, which are of course weakly monotone, and negative monotone
manifolds.  The latter are weakly monotone if and only if $N\geq n-2$.
\end{Remark}

\subsubsection{Quantum (co)homology}
\label{sec:QH}
The total Floer homology $\HF_*(H)$, equipped with the pair-of-pants
product, is an algebra over the Novikov ring $\Lambda$. This algebra
is isomorphic to the (small) \emph{quantum homology} $\HQ_*(M)$; see,
e.g., \cite{MS}. On the level of $\Lambda$-modules, we have
\begin{equation}
\label{eq:floer-qh}
\HQ_*(M)=\H_*(M)\otimes \Lambda
\end{equation}
with the tensor product grading. Thus, $|u\otimes e^A|=
|u|+I_{c_1}(A)$, where $u\in \H_*(M)$ and $A\in\Gamma$.  The
isomorphism between $\HF_*(H)$ and $\HQ_*(M)$ is defined via the
PSS-homomorphism; see \cite{PSS} or \cite{MS,U2}.  Alternatively, it
can be obtained from a homotopy of $H$ to an autonomous $C^2$-small
Hamiltonian (under slightly more restrictive conditions than weak
monotonicity, \cite{HS}) or with a somewhat different definition of
the total Floer homology (as the limit of $\HF^{(a,\,b)}_*(H)$ as
$a\to-\infty$ and $b\to \infty$, \cite{Ono:AC}).

The \emph{quantum product} $u*v$ of two elements $\H_*(M)$ is defined
as
\begin{equation}
\label{eq:qp}
u*v=\sum_{A\in\Gamma} (u*v)_A \,e^{A},
\end{equation}
where the class $(u*v)_A\in \H_*(M)$ is determined by the condition
that
$$
(u*v)_A\circ w=\GW^M_{A,3}(u,v,w)
$$
for all $w\in \H_*(M)$. Here $\circ$ denotes the intersection number
and $\GW^M_{A,3}$ is the corresponding Gromov--Witten invariant; see
\cite{MS}.

Note that $(u*v)_0=u\cap v$, where $\cap$ stands for the cap 
product and $u$ and $v$ are ordinary homology classes. Furthermore,
$$
|u*v|=|u|+|v|-2n
$$
and
\begin{equation}
\label{eq:qp-deg}
|(u*v)_A|=|u|+|v|-2n-I_{c_1}(A).
\end{equation}
Also observe that $I_\omega(A)=-\omega(A)<0$ whenever $A\neq 0$ can be
represented by a holomorphic curve. Thus, in \eqref{eq:qp}, it
suffices to limit the summation to the negative cone $I_\omega(A)\leq
0$. In particular, in the setting of Example \ref{ex:Novikov}, we can
write
$$
u*v=u\cap v+\sum_{k>0} (u*v)_k \, q^{k}.
$$
Here, $|(u*v)_k|=|u|+|v|-2n \pm 2Nk$ when $N<\infty$, with the
positive or negative sign depending on whether $M$ is positive or
negative monotone.  This sum is finite.

The product $*$ extends to a $\Lambda$-linear, associative,
graded-commutative product on $\HQ_*(M)$. The fundamental class $[M]$
is the unit in the algebra $\HQ_*(M)$. Thus, $qu=(q[M])*u$, where
$q\in\Lambda$ and $u\in \H_*(M)$, and $|qu|=|q|+|u|$. By the very
definition of $\HQ_*(M)$, the ordinary homology $\H_*(M)$ is
canonically embedded in $\HQ_*(M)$. The group of symplectomorphisms
acts on the algebra $\HQ_*(M)$ via its action on $\H_*(M)$ and,
clearly, symplectomorphisms isotopic to $\id$ act
trivially.

\begin{Example} 
\label{ex:cpn}
Let $M=\CP^n$. Then $N=n+1$ and, in the notation of Example
\ref{ex:Novikov}, $\HQ_*(\CP^n)$ is the quotient of
$\F[u]\otimes\Lambda$, where $u$ is the generator of
$\H_{2n-2}(\CP^n)$, by the ideal generated by the relation
$u^{n+1}=q[M]$.  Thus, $u^k=u\cap \ldots\cap u$ ($k$ times) when
$0\leq k\leq n$ and $[\pt]*u=q[M]$. For further examples of
calculations of quantum homology and relevant references we refer the
reader to, e.g., \cite{MS}.
\end{Example}

The quantum cohomology $\HQ^*(M)$ is defined by dualizing the entire
construction. We have $\HQ^*(M)=\H^*(M)\otimes \Lambda'$ as graded
modules over the Novikov ring $\Lambda'$, which is somewhat different
from $\Lambda$. The product of two ordinary cohomology classes is
obtained by taking the product of their Poincar\'e dual homology
classes $u$ and $v$ and then taking the Poincar\'e duals of the
coefficients $(u*v)_A$. Some care is needed in the definition of
$\Lambda'$. Namely, $\Lambda'$ is the completion of the group ring
$\F[\Gamma]$ ``in the opposite direction'', i.e., using the valuation
$-I_\omega$. (See, e.g., \cite[Remark 11.1.16]{MS}, for further
details.)  For our purposes, essentially for purely notational
reasons, it is convenient to swap the roles of $q$ and $q^{-1}$ in the
identification of the Novikov ring with the ring of Laurent series
(see Example \ref{ex:Novikov}). Thus, in cohomology, $|q|=2N$.

\begin{Remark}
  Note in conclusion that the definition of the Novikov ring $\Lambda$
  adapted in this paper is by no means standard in the context of
  quantum homology, although it is a natural choice as
  far as Floer homology is concerned.  Monograph \cite{MS} (see, in
  particular, Section 11.1) offers an extensive account of a variety of
  choices of the Novikov ring.
\end{Remark}

\subsection{Action selectors}
The theory of Hamiltonian \emph{action selectors} or \emph{spectral
  invariants}, as they are usually referred to, was developed in its
present Floer--theoretic form in \cite{Oh,Sch} although the first
versions of the theory go back to \cite{HZ,Vi:gen}. Here we briefly
recall the main results of the theory essential for our proofs, mainly
following \cite{GG:gaps}.

Let $M$ be a closed symplectic manifold and let $H$ be a Hamiltonian
on $M$. We assume that $M$ is rational -- this assumption greatly
simplifies the theory (cf.\ \cite{U1}) and is obviously satisfied for
monotone or negative monotone manifolds.

The \emph{action selector} $\s_v$ associated with a non-zero class $v\in
\HQ_{*}(M;\F) \cong \HF_*(H)$ is defined as
$$
\s_v(H)= \inf\{ a\in \R\ssminus \CS(H)\mid v\in \im(i^a)\}
=\inf\{ a\in \R\ssminus \CS(H)\mid j^a(v)=0\},
$$
where $i^a\colon \HF_*^{(-\infty,\,a)}(H)\to \HF_*(H)$ and
$j^a\colon \HF_*(H)\to\HF_*^{(a,\, \infty)}(H)$ are the natural
``inclusion'' and ``quotient'' maps.
Then $\s_v(H)>-\infty$ as is easy to see; \cite{Oh}. 

The action selector $\s_v$ has the following properties:

\begin{itemize}

\item[(AS1)] Normalization: $\s_{[M]}(H)=\max H$ if $H$ is autonomous
  and $C^2$-small.
             
\item[(AS2)] Continuity: $\s_v$ is Lipschitz in $H$ in the
  $C^0$-topology.

\item[(AS3)] Monotonicity: $\s_v(H)\geq \s_v(K)$ whenever $H\geq K$
  pointwise.

\item[(AS4)] Hamiltonian shift: $\s_v(H+a(t))=\s_v(H)+\int_0^1a(t)\,dt$,
  where $a\colon S^1\to\R$.

\item[(AS5)] Symplectic invariance: $\s_v(H)=\s_{\varphi^{-1}_*(v)}(\varphi^* H)$ for any
  symplectomorphism~$\varphi$.

\item[(AS6)] Homotopy invariance: $\s_v(H)=\s_v(K)$ when
  $\varphi_H=\varphi_K$ in the universal covering of the group of
  Hamiltonian diffeomorphisms and both $H$ and $K$ are normalized to
  have zero mean.

\item[(AS7)] Triangle inequality or sub-additivity: $\s_{v*u}(H\sharp
  K)\leq\s_v(H)+\s_u(K)$.

\item[(AS8)] Spectrality: $\s_v(H)\in \CS(H)$.  More specifically,
  there exists a capped one-periodic orbit $\bx$ of $H$ such that
  $\s_v(H)=\CA_H(\bx)$.

\item[(AS9)] Ljusternik--Schnirelman inequality:
  $\s_{v*u}(H)<\s_v(H)$, whenever one-periodic orbits of $H$ are
  isolated and $u\in\HQ_{*<2n}(M)$.

\end{itemize}

This list of the properties of $\s$ is far from exhaustive, but it is
more than sufficient for our purposes.  It is worth emphasizing that the
rationality assumption plays an important role in the proofs of the
homotopy invariance and spectrality; see \cite{Oh,Sch} and also
\cite{EP} for a simple proof. (The latter property also holds in
general for non-degenerate Hamiltonians. This is a non-trivial result;
\cite{U1}.)  The Ljusternik--Schnirelman inequality, (AS9), is
established in \cite[Proposition 6.2]{GG:gaps}. Finally note that for
the triangle inequality to hold one has to work with a suitable
definition of the pair-of-pants product in Floer homology; cf.\
\cite{AS,U2}. We refer the reader to \cite{U2} for a very detailed
treatment of action selectors.

\subsection{Carrier of the action selector} 
\label{sec:carriers}
When $H$ is non-degenerate, the action selector $\s_v$ can also be
evaluated as
$$
\s_v(H)=\inf_{[\sigma]=v}\CA_H(\sigma),
$$
where we set 
$$
\CA_H(\sigma)=\max\{\CA_H(\bx)\mid \sigma_{\bx} \neq 0\}\text{ for }
\sigma=\sum\sigma_{\bx} \bx\in\CF_{|v|}(H).
$$
The infimum here is obviously attained when $M$ is rational.  Hence,
there exists a cycle $\sigma=\sum\sigma_{\bx} \bx\in\CF_{|v|}(H)$,
representing $v$, such that $\s_v(H)=\CA_H(\bx)$ for an orbit $\bx$
entering $\sigma$. In other words, $\bx$ maximizes the action on
$\sigma$ and the cycle $\sigma$ minimizes the action over all cycles
in the homology class $v$. We call such an orbit $\bx$ a
\emph{carrier} of the action selector. Note that this is a stronger
requirement than just the equality $\s_v(H)=\CA_H(\bx)$. A carrier is
not in general unique, but it becomes unique when all one-periodic
orbits of $H$ have distinct action values.

Our next goal is to recall a generalization of this definition,
following \cite{GG:neg-mon}, to the case where one-periodic orbits
of $H$ are isolated but possibly degenerate. Under a $C^2$-small,
non-degenerate perturbation $\tH$ of $H$, every such orbit $x$ splits
into several non-degenerate orbits, which are close to
$x$. Furthermore, a capping of $x$ naturally gives rise to a capping
of each of these orbits.

\begin{Definition}
\label{def:carrier}
A capped one-periodic orbit $\bx$ of $H$ is a \emph{carrier} of the
action selector $\s_v$ for $H$ if there exists a sequence of
$C^2$-small, non-degenerate perturbations $\tH_i\stackrel{C^2}{\to} H$
such that one of the capped orbits which $\bx$ splits into is a
carrier for $\tH_i$. An orbit (without capping) is said to be a
carrier if it turns into one for a suitable choice of capping.
\end{Definition}

It is easy to see that a carrier necessarily exists, provided that $M$
is rational and all one-periodic orbits of $H$ are isolated. As in the
non-degenerate case, a carrier is of course not unique in general --
different choices of sequences $\tH_i$ and different choices of a
carrier for $\tH_i$ can lead to different carriers. However, it
becomes unique when all one-periodic orbits of $H$ have distinct
action values. In other words, under the latter requirement, the
carrier is independent of the choice of the sequence $\tH_i$ and the
choice of the carrier for $\tH_i$.

Picking a carrier for every $v\in \HQ_*(M)$, we obtain a map, also
referred to as a carrier,
$$
\bPhi_H\colon \HQ_*(M)\setminus\{0\}\to\bPP
$$
sending $v$ to the carrier for $\s_v$. This map, of course, is not
unique unless $H$ has distinct action values.

Let us assume now that $M$ is monotone or negative monotone with
$0<N<\infty$, i.e., $\left< c_1(TM),\pi_2(M)\right>\neq 0$ in $\R$,
and $\lambda\neq 0$. Thus $\Gamma\cong\Z$ and $|q|<0$, where we use
the notation from Example \ref{ex:Novikov}.

Clearly, when $H$ has distinct action values, $\bPhi$ is automatically
equivariant with respect to recapping:
\begin{equation}
\label{eq:q-equiv}
\bPhi_H(qv)=\bPhi_H(v)\#A, \text{ where } q=e^A.
\end{equation}

We claim that there is always a recapping-equivariant carrier
$\bPhi$, i.e., a carrier satisfying \eqref{eq:q-equiv}. Indeed, we can
pick $\bPhi$ on $\HQ_d(M)$ for all $d$ in any degree range of length
$2N$ (for instance, $[0, 2N-1]$) and then extend it to the entire quantum homology
``by periodicity'', i.e., by imposing \eqref{eq:q-equiv} on $\bPhi$.
 
A carrier gives rise to a map, also referred to as a carrier,
$$
\Phi\colon \HQ_*(M)\setminus\{0\}\to\PP
$$
forgetting the capping. Clearly, $\Phi$ is recapping--invariant, i.e.,
$\Phi_H(qv)=\Phi_H(v)$, when $\bPhi$ is recapping-equivariant.

\begin{Remark}
\label{rmk:carrier}
Note that, as an immediate consequence of the definition of the
carrier and continuity of the action and the mean index, we have
\begin{equation}
\label{eq:carrier}
\s_v(H)=\CA_H(\bx)\text{ and } |v|-2n\leq \Delta_H(\bx) \leq |v|,
\end{equation}
where $\bx=\bPhi(v)$, and the inequalities are strict when $x$ is
weakly non-degenerate. Furthermore, the local Floer homology of $H$ at
$\bx$ in degree $|v|$ is non-trivial: $\HF_{|v|}(H,\bx)\neq 0$. (This
fact is established in \cite{GG:neg-mon} for $v=[M]$; but the argument
applies to other homology classes word-for-word.  We refer the reader
to, e.g., \cite{GG:gaps,GG:gap} for a detailed discussion of the local
Floer homology.)  Finally note that under our requirements on $M$ it
is not hard to show that $\bPhi$ can be chosen so that $\bPhi_H(\alpha
v)=\bPhi_H(v)$ for all $\alpha\neq 0$ in $\F$.
\end{Remark}

\begin{Remark}
  We finish this discussion with one minor, fairly standard, technical
  point; cf.\ \cite{GG:neg-mon}. Namely, recall that the Floer complex
  of a non-degenerate Hamiltonian $H$ depends not only on $H$ but also
  on an auxiliary structure $J$, e.g., an almost complex structure
  when $M$ is weakly monotone. Moreover, the complex is defined only
  when suitable regularity requirements are met. As a consequence, an
  action selector carrier is in reality assigned to the pair $(H,J)$ rather
  than to just a Hamiltonian $H$ in both the non-degenerate and
  degenerate cases. Thus, in Definition \ref{def:carrier}, we tacitly
  assumed the presence of an auxiliary structure $J$ in the background
  and that the regularity requirements are satisfied for the sequence
  of perturbations. This can be achieved by either considering regular
  pairs $(\tH_i,J_i)$ with $J_i\to J$ or even by setting $J_i=J$.
\end{Remark}

\section{Proof of Theorem \ref{thm:relations}}
\label{sec:proof-relations}

Let, as in Theorem \ref{thm:relations}, $\varphi_H$ be a Hamiltonian
diffeomorphism with finitely many periodic orbits of a monotone
symplectic manifold $M$. (Recall that we can assume that $0<N<\infty$,
i.e., $\left<c_1(TM),\pi_2(M)\right>\neq 0$; for otherwise $\varphi_H$
has infinitely many periodic points; see \cite {GG:gaps,He:irr} and
also \cite{Gi:conley,Hi}). Recall also that $\PP$ (and $\bPP$) stand
for the collection of (capped) one-periodic orbits of $\varphi_H$.

Although the theorem is stated in terms of cohomology, we find working
with homology  more intuitive at this stage. When translated to
homology, the cohomological product decomposition \eqref{eq:product}
retains the same form
\begin{equation}
\label{eq:product2}
u_0*u_1*\cdots *u_{\ell} =q^\nu u_0,
\end{equation}
where now all $u_j$ are in $\HQ_*(M)$, the classes
$u_1,\ldots,u_{\ell}$ have degree less than $2n$, and
\begin{equation}
\label{eq:deg2}
2n(\ell -1) -|u_1|-\cdots -|u_{\ell-1}| < 2N.
\end{equation}
As in \eqref{eq:product}, we have $\nu>0$.

In Case (ii) we simply have $u^d\neq 0$, where $|u|<2n$.  Recall
also that now, since we are using homology, $|q|=-2N$.

\subsection{Case (i)}
\label{sec:case-i}
Set 
$$
\begin{array}{lll}
v_0 & := & u_0, \\
v_1 & := & v_0* u_1, \\
v_2 & := & v_1* u_2, \\ 
\multicolumn{3}{c}\dotfill \\
v_{\ell-1} & := & v_{\ell-2}* u_{\ell-1}.
\end{array}
$$
It is convenient to extend the sequence $v_j$ with $0\leq j\leq \ell-1$ in
both directions by periodicity, using \eqref{eq:product2} in the
definition of $v_\ell$. Namely, we set
$$
\begin{array}{lclcl}
v_\ell &:= &  v_{\ell-1}* u_\ell &=& q^\nu v_0,\\
 v_{\ell+1} & := & v_\ell*u_1 & = & q^\nu v_1, \\
v_{\ell+2}  & := & v_{\ell+1}*u_2 & = &q^\nu v_2, \\
\multicolumn{5}{c}\dotfill \\
\end{array}
$$
and
$$
v_{-1}:=q^{-\nu}v_{\ell-1}= v_{-2}*u_{\ell-1},
$$
where
$$
v_{-2}:=q^{-\nu}v_{\ell-2}=v_{-3}*u_{\ell-2},
$$
etc. As a result, we have a sequence $v_j$ such that
\begin{equation}
\label{eq:periodicity}
v_{j+\ell}=q^\nu v_j
\end{equation}
for some $\nu>0$, and 
\begin{equation}
\label{eq:product3}
v_{j+1}=v_j*w_{j+1},
\end{equation}
for some $w_j\in \HQ_{*<2n}(M)$.
(Here $w_{j+1}=u_{j+1}$ for $j=0,\ldots, \ell-1$ and then again the
sequence $w_j$ extends by periodicity.)  We call such a sequence
$\{v_j\}$ a \emph{ladder} and $\ell$ the \emph{length} of the
ladder. Let us denote the entire ladder by $L$ and its segment
$\{v_0,\ldots,v_{\ell-1}\}$ by $V$.

It is important for what follows that in addition to the requirements
\eqref{eq:periodicity} and \eqref{eq:product3} we also have
\eqref{eq:deg2} satisfied, i.e., in terms of the ladder,
\begin{equation}
\label{eq:degs}
|v_0|>|v_1|>\cdots >|v_{\ell-1}|>|v_0|-2N.
\end{equation}

Clearly, a ladder $L=\{v_j\}$ is strictly ordered by the index $|v_j|$
(or, to be more precise, the degree), since $|w_j|<2n$. (This fact is
also incorporated in \eqref{eq:degs}.) Furthermore,
for any Hamiltonian
$K$ with isolated fixed points, $L$ is strictly ordered by the action, i.e.,
\begin{equation}
\label{eq:action-ordering}
\s_{v_{j}}(K)> \s_{v_{j+1}}(K) \text{ for all $j$}.
\end{equation}
Here the non-strict inequality follows immediately from
\eqref{eq:product3} and the sub-additivity of the action selector,
(AS7), and holds for any $K$. The strict inequality requires the fixed
points of $K$ to be isolated and is a consequence of the
Ljusternik--Schnirelman inequality (AS9), cf.\ \cite[Proposition
6.2]{GG:gaps}. It is essential that these two orderings of $L$
coincide.

We claim that there exists a sequence of prime iterations
$k_i\to\infty$ and a sequence of recapping--equivariant action selector
carriers $\bPhi_{H^{\sharp k_i}}$ such that all maps $\Phi_{H^{\sharp
    k_i}}|_V$ (or, equivalently, $\Phi_{H^{\sharp
    k_i}}|_L$) coincide, i.e.,
\begin{equation}
\label{eq:image}
\Phi_{H^{\sharp k_1}}(v_j) = \Phi_{H^{\sharp k_2}}(v_j)
 = \cdots = \Phi_{H^{\sharp k_i}}(v_{j}) = \cdots
\end{equation}
for all $j=0,\ldots, \ell-1$ and hence, by periodicity of $L$, for all
$j\in\Z$.  In other words, all carriers $\bPhi_{H^{\sharp k_i}}$
assume the same value, up to recapping, on each class $v_j$ in $L$.

Indeed, note that since $\varphi_H$ has finitely many periodic orbits,
for every sufficiently large prime $p_i$, every $p_i$-periodic point
is in fact a fixed point. Thus we start with a sequence $p_i$ of all
sufficiently large primes.  The existence of the subsequence $k_i$ in
this sequence follows immediately from the pigeonhole principle and
the existence of a recapping-equivariant carrier for any
Hamiltonian. Indeed, the collection of maps from $V$ to $\PP$ is
finite. Hence, there is only a finite number of possible maps 
$\Phi_{H^{\sharp p_i}}|_V$.

\begin{Remark}
\label{rmk:finiteness}
When the ground field $\F$ is finite, a similar argument shows that
every infinite sequence of iterations contains an infinite subsequence
$k_i$ such that the carriers $\bPhi_{H^{\sharp k_i}}$ are
recapping--equivariant and all $\Phi_{H^{\sharp k_i}}$ are identically
equal to each other on $\HQ_*(M;\F)$. (It is not clear however whether
this would also be true when, for instance, $\F=\Q$.) In fact, for any
ground field $\F$, the argument applies to any finite collection of
non-zero elements in $\HQ_*(M;\F)$ in place of $V$.
\end{Remark}

Before we continue the proof, let us introduce some notation. Namely, set
$$
\bPhi_i=\bPhi_{H^{\sharp k_i}}|_L\colon L\to \bPP .
$$
These maps are one-to-one by \eqref{eq:action-ordering}.  Also note
that the maps $\Phi_i=\Phi_{H^{\sharp k_i}}|_L$ agree, due to
\eqref{eq:image}, and we denote them by $\Phi$ in what follows.

Next we claim that all fixed points in the image $\Phi(L)=\Phi(V)$
have the same augmented action.  Thus let us pick two points
$x$ and $y$ in the image. Our goal is to show that
$$
\tA_{H}(x)=\tA_H(y)
$$
or, equivalently, once arbitrary cappings of $x$ and $y$ are fixed,
that
\begin{equation}
\label{eq:action-index}
\A_{H}(\bx)-\A_H(\by)=\frac{\lambda}{2}\big(\Delta_{H}(\bx)-\Delta_H(\by)\big),
\end{equation}
where $\lambda=\lambda_0/N$ is the monotonicity constant.

\begin{Remark}
\label{rmk:important-points}
It is essential for the proof of Case (ii) that the reasoning
establishing \eqref{eq:action-index} only uses the periodicity
property \eqref{eq:periodicity} and the action ordering property
\eqref{eq:action-ordering}. Other than relying on these two results,
it is independent of the fact that $L$ is a ladder: the product
requirement \eqref{eq:product3} does not enter the argument directly,
but only via the proof of \eqref{eq:action-ordering}. Furthermore, to
prove in the non-degenerate case that all orbits in $\Phi(V)$ are
distinct we will only use \eqref{eq:degs}.
\end{Remark}

Consider the capped orbits $\bx^{k_i}$ and $\by^{k_i}$. These orbits
need not be in the image of $\bPhi_i$ unless $\nu=1$. We denote the
orbits in the image, closest to $\bx^{k_i}$ and $\by^{k_i}$, by
$\bx_i$ and $\by_i$, respectively. If the closest orbit is not unique
-- there can be two -- we pick it in an arbitrary way. (When $\nu=1$,
we have $\bx_i=\bx^{k_i}$ and $\by_i=\by^{k_i}$.) It readily follows
from the $q^\nu$-periodicity of $L$ that
\begin{equation}
\label{eq:adjustment-action}
\big|\A_{H^{\sharp k_i}}(\bx_i)-\A_{H^{\sharp k_i}}(\bx^{k_i})\big|\leq
\frac{\lambda_0\nu}{2}
\text{ and }
\big|\A_{H^{\sharp k_i}}(\by_i)-\A_{H^{\sharp k_i}}(\by^{k_i})\big|\leq
\frac{\lambda_0\nu}{2}
\end{equation}
and
\begin{equation}
\label{eq:adjustment-index}
\big|\Delta_{H^{\sharp k_i}}(\bx_i)-\Delta_{H^{\sharp k_i}}(\bx^{k_i})\big|\leq
N\nu
\text{ and }
\big|\Delta_{H^{\sharp k_i}}(\by_i)-\Delta_{H^{\sharp k_i}}(\by^{k_i})\big|\leq
N\nu .
\end{equation}
Indeed, among all cappings of $x$ or $y$ those coming from $L$ via
$\bPhi_i$ for any $i$ occur periodically at least once within any
interval of $\nu$ cappings.

Let us estimate the number $m_i$ of capped orbits in $\bPhi_i(L)$
between $\bx_i$ and $\by_i$ using the action and index orderings. Note
that the map $\bPhi_i$ and even its image (unless $\nu=1$)
depend on $i$ and so does $m_i$. However, the two orderings agree for
every $i$, and the results must be the same whether we use the index
or action ordering. Without loss of generality, we may assume that
$\A_{H}(\bx)>\A_{H}(\by)$ and hence $\A_{H^{\sharp k_i}}(\bx^{k_i})>
\A_{H^{\sharp k_i}}(\by^{k_i})$.

Every class in $V$ contributes an orbit occurring periodically in $L$.
From the action perspective, the period is $\lambda_0\nu$ by
\eqref{eq:product2}. Thus, we have
\begin{eqnarray*}
m_i &=& \frac{|\Phi(V)|}{\lambda_0\nu}
\big(\A_{H^{\sharp k_i}}(\bx_i)-\A_{H^{\sharp k_i}}(\by_i)\big)+\const\\
&=& \frac{|\Phi(V)|}{\lambda_0\nu}
\big(\A_{H^{\sharp k_i}}(\bx^{k_i})-\A_{H^{\sharp k_i}}(\by^{k_i})\big)+\const,
\end{eqnarray*}
where we use \eqref{eq:adjustment-action} to pass to the second
equality.  Henceforth $\const$ stands for a term bounded from
above and below by a constant independent of $i$, which can be fixed
throughout the entire argument; the actual value of this term is
immaterial and may vary from formula to formula or from one part of a
formula to another.

From the index perspective, the period is $2N\nu$ and we have, using
now \eqref{eq:adjustment-index},
\begin{eqnarray*}
m_i &=& \frac{|\Phi(V)|}{2N\nu}
\big(\Delta_{H^{\sharp k_i}}(\bx_i)-\Delta_{H^{\sharp k_i}}(\by_i)\big)+\const\\
&=& \frac{|\Phi(V)|}{2N\nu}
\big(\Delta_{H^{\sharp k_i}}(\bx^{k_i})-\Delta_{H^{\sharp k_i}}(\by^{k_i})\big)+\const.
\end{eqnarray*}

Equating the results, we see after a simple algebraic manipulation
that
$$
\frac{1}{\lambda_0}
\big(\A_{H^{\sharp k_i}}(\bx^{k_i})-\A_{H^{\sharp k_i}}(\by^{k_i})\big)
=\frac{1}{2N}
\big(\Delta_{H^{\sharp k_i}}(\bx^{k_i})-\Delta_{H^{\sharp
    k_i}}(\by^{k_i})\big)+\const.
$$
Finally, dividing by $k_i$ and passing to the limit as $k_i\to\infty$,
we arrive at \eqref{eq:action-index} since the action and the mean
index are homogeneous with respect to the iteration.

Finally set 
$$
x_j=\Phi(v_j)\text{ for } j=0,\ldots,\ell-1.
$$
To finish the proof, we only need to show that these orbits are
distinct, i.e., $|\Phi(V)|=\ell$. Let us cap the orbits using say
$\bPhi_1$, i.e., by setting
$\bx_j=\bPhi_1(v_j)$. 
Then by \eqref{eq:action-ordering}, we have
$$
\A_{H^{^\sharp k_1}}(\bx_0)>
\A_{H^{^\sharp k_1}} (\bx_1)>\cdots> 
\A_{H^{^\sharp k_1}} (\bx_{\ell-1})>
\A_{H^{^\sharp k_1}} (\bx_0)-\lambda_0\nu.
$$
When $\nu=1$, this immediately implies that no two orbits
$\Phi(V)$ are recappings of each
other. When $\nu>1$, we use the Conley--Zehnder index -- hence the
non-degeneracy assumption -- rather than the action to distinguish the
orbits. Namely, recall that $\MUCZ(\bx_j)=|v_j|-n$. Thus,
\eqref{eq:degs} is equivalent to
$$
\MUCZ(\bx_0)>\MUCZ(\bx_1)>\cdots > \MUCZ(\bx_{\ell-1})>\MUCZ(\bx_0)-2N,
$$
and, as a consequence, $\big|\MUCZ(\bx_j)-\MUCZ(\bx_l)\big|<2N$.
Hence, all orbits in $\Phi(V)$ are distinct.

\subsection{Case (ii)} The idea of the proof is that for the previous
argument to go through we do not need the product decomposition
\eqref{eq:product2} to hold literally. It is in fact sufficient to
have action selector carriers taking the same value (up to a capping)
on the left and right hand sides of \eqref{eq:product2} for the
sequence of iterations $k_i$. The proof shares many common elements
with the reasoning in Case (i) and below we only detail the necessary
changes.

Consider the finite collection 
$$
U=\{u, u^2, \ldots, u^d\}
$$
of non-zero elements in $\HQ_{*<2n}(M)$, where
\begin{equation}
\label{eq:|u^r|}
|u^r|=2n-(2n-|u|)r.
\end{equation}
By arguing exactly as in the proof of Case (i), it is easy to find a
sequence of prime iterations $k_i$ and a sequence of
recapping--equivariant action selector carriers $\bPhi_{H^{\sharp k_i}}$ such
that the maps $\Phi_{H^{\sharp k_i}}$ agree on $U$; cf.\ Remark
\ref{rmk:finiteness}. Let us denote the resulting map $U\to\PP$ by
$\Phi$. 

Since $|U|=d$ is sufficiently large, e.g.,
$$
|U|=\bigg\lceil \frac{2N|\PP|}{2n-|u|}\bigg\rceil+1,
$$
there exist $s_-$ and $s_+$, both in the range $[1,\, d]$, such that
$s_+>s_- +2N/(2n-|u|)$ or,
equivalently,
$$
\big| u^{s_-} \big|-\big| u^{s+}\big|> 2N
$$
and
$$
\Phi(u^{s_-})=\Phi(u^{s_+}).
$$
This is again an immediate consequence of the pigeonhole
principle. Indeed, otherwise the number of classes mapped to any point
in $\PP$ would not exceed $2N/(2n-|u|)$, and hence we would have
$|U|\leq 2N|\PP|/(2n-|u|)$.

Let, as in Example \ref{ex:Novikov}, $A$ be the generator of $\Gamma$ 
such that $I_{c_1}(A)=-2N$, i.e., $q=e^A$.
For every $k_i$, 
\begin{equation}
\label{eq:match}
\bPhi_{H^{\sharp k_i}}(u^{s_+})=
\bPhi_{H^{\sharp k_i}}(u^{s_-})\# (\nu A),
\end{equation}
where $\nu=(s_+ -s_-)(2n-|u|)/2N$ since, due to the non-degeneracy
assumption,
$$
\MUCZ\big(\bPhi_{H^{\sharp k_i}}(u^{s_+})\big)=|u^{s_+}|-n \text{ and }
\MUCZ\big(\bPhi_{H^{\sharp k_i}}(u^{s_-})\big)=|u^{s_-}|-n.
$$
In particular, $\nu$ is independent of $k_i$. (One can bypass this
reference to non-degeneracy by observing that there are only finitely
many possible values of $\nu$, as a simple mean index argument shows,
and then by passing to a subsequence of iterations.)

Let us set $\ell=\big\lfloor 2N/(2n-|u|)\big\rfloor$. Consider the finite sequence $V$
formed by $\ell$ classes
$$
\begin{array}{lllll}
v_0 &:=& u^{s_-}, & & \\
v_1 &:=& u^{s_- +1} &= &v_0*u,\\
\multicolumn{5}{c}\dotfill \\
 v_{\ell-1} & := & u^{s_- +\ell-1} &=& v_{\ell-2}*u,
\end{array}
$$
which we extend in both directions by $q^\nu$-periodicity to have
\eqref{eq:periodicity} satisfied. In other words, we set
$$ 
v_\ell=q^\nu v_0,\, v_{\ell+1}=q^\nu v_1,\,\ldots 
$$
and 
$$
v_{-1}=q^{-\nu}v_{\ell-1},\,
v_{-2}=q^{-\nu}v_{\ell-2},\, \ldots .
$$

Note that 
\begin{equation}
\label{eq:u^s+}
u^{s_+}=v_{\ell-1}*u^{s_+ -s_- -\ell+1}, \text{ where } s_+ -s_- -\ell+1>0
\end{equation}
due to our choice of $\ell$. Furthermore, for all $i$, we
have 
\begin{equation}
\label{eq:u^s+-v_l}
\s_{u^{s_+}}\big(H^{\sharp k_i}\big)=\s_{v_\ell}\big(H^{\sharp k_i}\big)
\end{equation}
by  \eqref{eq:match} and the definition of $v_{\ell}$.

The sequence $L=\{v_j\}$ is not a ladder. For
\eqref{eq:product3} need not hold for $j=\ell$.  However, we claim
that $L$ is still strictly ordered by the index and the action (for
the Hamiltonians $H^{\sharp k_i}$) and that the two orderings agree.

For the index ordering, this immediately follows from \eqref{eq:|u^r|}
and our choice of $\ell$.  To see that $L$ is strictly ordered by the
action, i.e., \eqref{eq:action-ordering} holds for $K=H^{\sharp k_i}$,
we argue as follows. For $j$ in the range from $0$ to $\ell-2$,
exactly as in Case (i) this follows from (AS9) for any Hamiltonian
with isolated fixed points. For $j=\ell-1$, we have
\begin{eqnarray*}
\s_{v_\ell}\big(H^{\sharp k_i}\big) &=& 
\s_{v_0}\big(H^{\sharp k_i}\big)-\nu\lambda_0\\
&=& 
\s_{u^{s_-}}\big(H^{\sharp k_i}\big)-\nu\lambda_0\\
&=& 
\s_{u^{s_+}}\big(H^{\sharp k_i}\big)\\
&<& \s_{v_{\ell-1}}\big(H^{\sharp k_i}\big).
\end{eqnarray*}
Here the third equality follows from \eqref{eq:match}. The ultimate
inequality is again a consequence of the Ljusternik--Schnirelman
inequality, (AS9), and \eqref{eq:u^s+} and \eqref{eq:u^s+-v_l}.
Hence, the required inequalities \eqref{eq:action-ordering} hold for
$j=0,\,\ldots,\,\ell-1$ and, by periodicity, for all $j$ and
$K=H^{\sharp k_i}$.

Finally note that \eqref{eq:degs} is satisfied by the definition of
$L$ and $\ell$.  With these observations in mind, the proof is finished
exactly in the same way in Case (i); cf.\ Remark
\ref{rmk:important-points}.

\section{Proof of Theorem \ref{thm:conley}}
\label{sec:proof-conley}

The proof follows roughly the same line of reasoning as the argument
in \cite{GG:neg-mon} with some changes in the general logic of the
proof.  Throughout this section, we will use the notation and
conventions from Sections \ref{sec:prelim} and
\ref{sec:proof-relations}.

Arguing by contradiction, assume that all periodic orbits of
$\varphi_H$ with sufficiently large period are iterated. In
particular, for every large prime $k$ the set of
contractible $k$-periodic orbits is naturally identified with the set
$\PP$ of contractible one-periodic orbits.  Then there exists a
sequence of primes $k_i\to\infty$ and a sequence of
recapping-equivariant action carriers $\bPhi_{H^{\sharp k_i}}$ such
that
$$
\Phi_{H^{\sharp k_1}}\big([M]\big)=\Phi_{H^{\sharp k_2}}\big([M]\big)=\cdots .
$$
Indeed, for a large prime $k$, the carrier $\Phi_{H^{\sharp k}}$ takes
values in $\PP$ and the assertion follows from the pigeonhole
principle. We denote the resulting fixed point by $x$. Thus, we have
$$
\Phi_{H^{\sharp k_i}}\big([M]\big)=x^{k_i}.
$$

Set
$$
\bPhi_{H^{\sharp k_i}}\big([M]\big)=\bx_{k_i}.
$$
Let us first focus on the capped orbit $\bx_{k_1}$. By Remark
\ref{rmk:carrier} and, in particular, \eqref{eq:carrier}, we have
$$
0\leq \Delta_{H^{\sharp k_1}}(\bx_{k_1})\leq 2n \text{ and }
 \A_{H^{\sharp k_1}}(\bx_{k_1})=\s_{[M]}\big(H^{\sharp k_1}\big) .
$$
We treat the cases where the mean index is zero and where it is
positive separately. These are the so-called the degenerate and weakly
non-degenerate cases of the Conley conjecture.
 
\subsubsection*{The degenerate case: $\Delta_{H^{\sharp
      k_1}}(\bx_{k_1})=0$} Note that we can, without loss of
generality, take $k_1$ arbitrarily large. More specificially, we can
assume that $k_1$ is so large that none of the Floquet multipliers,
different from one, of any one-periodic orbit $y\in \PP$ is a $k_1$-th
root of unity. In other words, using the terminology from
\cite{GG:gap}, $k_1$ is an \emph{admissible iteration} of
$\varphi_H$. Furthermore, we require $k_1$ to be large enough to
ensure that $\Delta_H(y)=0 \mod 2N$ whenever $k_1\Delta_H(y)=0\mod 2N$
for any $y\in\PP$. (Here we treat the mean index as an element of
$\R/2N\Z$, which is obviously well--defined, i.e., independent of the
capping.)

In particular, we have
$$
k_1\Delta_H(x)=\Delta_{H^{\sharp k_1}}(x^{k_1})=\Delta_{H^{\sharp k_1}}(\bx_{k_1})=0\mod 2N,
$$
and hence $\Delta_H(x)=0\mod 2N$.  As a consequence, $\Delta_H(\bx)=0$
for a suitable capping $\bx$ of $x$. With this capping, $\Delta_{H^{\sharp
    k_1}}\big(\bx^{k_1}\big)=\Delta_{H^{\sharp
    k_1}}\big(\bx_{k_1}\big) $ in $\Z$, and therefore
$\bx^{k_1}=\bx_{k_1}$, since $M$ is negative monotone.

We claim that $\bx$ is the so-called \emph{symplectically degenerate
  maximum} of $H$, i.e., $\HF_{2n}(H,\bx)\neq 0$ and
$\Delta_H(\bx)=0$; \cite {GG:gaps}. (See also
\cite{Gi:conley,GG:gap,He:irr} for the definition, a detailed
discussion and applications of this notion, which originates from
Hingston's proof of the Conley conjecture for tori; see \cite{Hi}.)
The vanishing of the mean index has already been established. On the
other hand, the local Floer homology of $\bx$ in degree $2n$ does not
vanish since
$$
\HF_{2n}(H,\bx)=\HF_{2n}\big(H^{\sharp k_1},\bx^{k_1}\big)\neq 0 .
$$
Here, the first equality is a consequence of the persistence of local
Floer homology for admissible iterations established in \cite{GG:gap}
and the second one follows from Remark~\ref{rmk:carrier}.

In the presence of a symplectically degenerate maximum, the Conley
conjecture (the existence of simple orbits with arbitrarily large
period) is proved for rational, and in particular negative-monotone,
symplectic manifolds in \cite{GG:gaps}. This concludes the proof of
Theorem \ref{thm:conley} in the degenerate case.

\subsubsection*{The weakly non-degenerate case: $\Delta_{H^{\sharp
      k_1}}(\bx_{k_1})> 0$}
Set $l_i=\lfloor k_i/k_1 \rfloor$ and let $r_i$ to be the remainder of dividing
$k_i$ by $k_1$, i.e., $l_ik_1+r_i=k_i$ and $0\leq r_i< k_1$. Define
$\nu_i$ by
\begin{equation}
\label{eq:nui}
\bx_{k_i}=\bx_{k_1}^{l_i}+\nu_i A,
\end{equation}
where $A$ is a generator of $\Gamma$.

Looking at the action values on $\bx_{k_i}$, we have,
\begin{eqnarray*}
\A_{H^{\sharp k_i}}(\bx_{k_i}) &=& \s_{[M]}\big(H^{\sharp k_i}\big) \\ 
&\leq & l_i\s_{[M]}\big(H^{\sharp k_1}\big)  +\s_{[M]}\big(H^{\sharp
  r_i}\big) \\
&\leq & l_i\s_{[M]}\big(H^{\sharp k_1}\big)  +\const,
\end{eqnarray*}
where, as in Section \ref{sec:case-i}, $\const$ stands for a constant
independent of $i$. Here, the first inequality follows from the
sub-additivity of the action selector.

On the other hand, by \eqref{eq:nui} and since the action is
homogeneous,
$$
\A_{H^{\sharp k_i}}(\bx_{k_i})=l_i \A_{H^{\sharp
    k_1}}(\bx_{k_1})+\nu_i I_\omega(A),
$$
and thus 
\begin{equation}
\label{eq:action-bound}
\nu_i I_\omega(A)\leq \const.
\end{equation}

Examining the mean indexes, we obtain in a similar vein that

$$
\Delta_{H^{\sharp k_i}}(\bx_{k_i}) =l_i \Delta_{H^{\sharp
    k_1}}(\bx_{k_1})+\nu_i I_{c_1}(A),
$$
where we used again \eqref{eq:nui} and the homogeneity of the mean
index. By the weak non-degeneracy assumption we have 
$\Delta_{H^{\sharp k_1}}(\bx_{k_1})>0$, and we conclude that
$$
\nu_i I_{c_1}(A)\to -\infty
$$ 
as $k_i\to\infty$. Therefore, since $M$ is negative monotone,
$$
\nu_i I_{\omega}(A)\to \infty,
$$ 
which is impossible due to \eqref{eq:action-bound}. This contradiction
completes the proof of the theorem.


\begin{thebibliography}{BEHW}


\bibitem[AS]{AS}
A. Abbondandolo, M. Schwarz,
Floer homology of cotangent bundles and the loop product, \emph{Geom.\
  Top.}, \textbf{14} (2010), 1569--1722.


\bibitem[AK]{AK}
D.V.  Anosov, A.B. Katok,  
New examples in smooth ergodic theory. Ergodic diffeomorphisms, (in
Russian), \emph{Trudy Moskov.\ Mat.\ Ob\v{s}\v{c}.}, \textbf{23}
(1970), 3--36.

\bibitem[Be]{Be}
A. Bertram,
Quantum Schubert calculus, \emph{Adv.\ Math.}, \textbf{128} (1997),
289--305.

\bibitem[Ek]{E}
I. Ekeland, 
Une th\'eorie de Morse pour les syst\`emes hamiltoniens convexes,
\emph{Ann.\ Inst.\ H. Poincar\'e Anal.\ Non Lin\'eaire}, \textbf{1}
(1984), 19--78.


\bibitem[EH]{EH}
I. Ekeland, H. Hofer, 
Convex Hamiltonian energy surfaces and their periodic trajectories,
\emph{Comm.\ Math.\ Phys.}, \textbf{113} (1987), 419--469.


\bibitem[EP]{EP} 
M. Entov, L. Polterovich, 
Rigid subsets of symplectic manifolds, \emph{Compos.\ Math.},
\textbf{145} (2009), 773--826.


\bibitem[Es]{Es}
J. Espina, 
On the mean Euler characteristic of contact manifold, Preprint 2010,
arXiv:1011.4364.

\bibitem[FK]{FK}
B. Fayad, A. Katok,
Constructions in elliptic dynamics, 
\emph{Ergodic Theory Dynam.\ Systems}, \textbf{24} (2004), 1477--1520.

\bibitem[FO]{FO}
K. Fukaya, K. Ono, 
Arnold conjecture and Gromov--Witten invariant,
\emph{Topology}, \textbf{38} (1999), 933--1048. 

\bibitem[FOOO]{FOOO} 
K. Fukaya, Y.-G. Oh, H. Ohta, K. Ono,
\emph{Lagrangian Intersection Floer Theory: Anomaly and Obstruction},
Parts I and II, AMS/IP Studies in Advanced Mathematics, vol. 46.1 and
46.2, American Mathematical Society, Providence, RI; International
Press, Somerville, MA, 2009.

\bibitem[Gi]{Gi:conley}
V.L. Ginzburg,
The Conley conjecture, \emph{Ann.\ of Math.}, \textbf{172} (2010), 1127--1180.

\bibitem[GG1]{GG:gaps}
V.L. Ginzburg, B.Z. G\"urel,
Action and index spectra and periodic orbits in Hamiltonian dynamics,
\emph{Geom.\ Topol.}, \textbf{13} (2009), 2745--2805.

\bibitem[GG2]{GG:gap}
V.L. Ginzburg, B.Z. G\"urel,
Local Floer homology and the action gap, \emph{J. Sympl.\ Geom.},
\textbf{8} (2010), 323--357.

\bibitem[GG3]{GG:neg-mon} 
V.L. Ginzburg, B.Z. G\"urel,
Conley conjecture for negative monotone symplectic manifolds,
\emph{Int.\ Math.\ Res.\ Not.\ IMRN}, 2011, doi:10.1093/imrn/rnr081.

 
\bibitem[GK]{GK}
V.L. Ginzburg, E. Kerman, 
Homological resonances for Hamiltonian diffeomorphisms and Reeb flows,
\emph{Int.\ Math.\ Res.\ Not.\ IMRN}, \textbf{2010}, 53–68. 

\bibitem[GGK]{GGK}
V. Guillemin, V. Ginzburg, Y. Karshon, 
\emph{Moment Maps, Cobordisms, and Hamiltonian Group Actions},
Mathematical Surveys and Monographs, \textbf{98}.  American
Mathematical Society, Providence, RI, 2002.

\bibitem[Hi]{Hi}
N. Hingston,
Subharmonic solutions of Hamiltonian equations on tori, 
\emph{Ann.\ of Math.}, \textbf{170} (2009), 525--560.

\bibitem[He]{He:irr}
D. Hein,
The Conley conjecture for irrational symplectic manifolds, Preprint
2009, arXiv:0912.2064; to appear in \emph{J. Sympl.\ Geom.}

\bibitem[HS]{HS}
H. Hofer, D. Salamon, 
Floer homology and Novikov rings, in \emph{The Floer memorial volume},
483--524, Progr.\ Math., 133, Birkh\"auser, Basel, 1995.

\bibitem[HWZ1]{HWZ} 
H. Hofer, K. Wysocki, E. Zehnder, 
SC-smoothness, retractions and new models for smooth spaces,
\emph{Discrete Contin.\ Dyn.\
Syst.},  \textbf{28}  (2010),  665--788. 

\bibitem[HWZ2]{HWZ2} 
H. Hofer, K. Wysocki, E. Zehnder, Applications of
polyfold theory I: The Polyfolds of Gromov-Witten Theory, Preprint
2011, arXiv:1107.2097.

\bibitem[HZ]{HZ}
H. Hofer, E. Zehnder,
\emph{Symplectic Invariants and Hamiltonian Dynamics}, Birk\"auser,
1994.

\bibitem[Ka]{Ka} 
R. Kaufmann, 
The intersection form in $H_*(\bar{M}_{0,n})$ and the explicit
K\"unneth formula in quantum cohomology, \emph{Internat.\ Math.\ Res.\
  Notices}, 1996, no.\ 19, 929--952.

\bibitem[LT]{LT}
G. Liu, G. Tian,
Floer homology and Arnold conjecture,  \emph{J. Differential Geom.},
\textbf{49}  (1998),  1--74. 

\bibitem[Lo]{Lo}
Y. Long,  \emph{Index Theory for Symplectic Paths with Applications}, 
Progress in Mathematics, 207. Birkh\"auser Verlag, Basel, 2002. 

\bibitem[MS]{MS}
D. McDuff, D. Salamon,
\emph{J-holomorphic Curves and Symplectic Topology}, Colloquium
publications, vol.\ 52, AMS, Providence, RI, 2004.

\bibitem[Oh]{Oh}
Y.-G. Oh, 
Construction of spectral invariants of Hamiltonian paths on closed
symplectic manifolds, in \emph{The breadth of symplectic and Poisson
geometry}, 525--570, Progr.\ Math., 232, Birkh\"auser, Boston, MA,
2005.

\bibitem[On]{Ono:AC}
K. Ono, 
On the Arnold conjecture for weakly monotone symplectic manifolds,
\emph{Invent.\ Math.}, \textbf{119} (1995), 519--537. 

\bibitem[PSS]{PSS}
S. Piunikhin, D. Salamon, M. Schwarz,
Symplectic Floer--Donaldson theory and quantum cohomology, in
\emph{Contact and Symplectic Geometry} (Cambridge, 1994), 171--200;
C.B. Thomas (Ed.), Publ.\ Newton Inst., 8, Cambridge Univ.\ Press,
Cambridge, 1996.

\bibitem[Po]{Po97}
L. Polterovich,
Hamiltonian loops and Arnold's principle, \emph{Topics in
singularity theory}, 181--187, Amer.\ Math.\ Soc.\ Transl.\ Ser.\ 2,
\textbf{180}, Amer.\ Math.\ Soc., Providence, RI, 1997.


\bibitem[Sa]{Sa}
D.A. Salamon,
Lectures on Floer homology, in \emph{Symplectic Geometry and
Topology}, Eds: Y. Eliashberg and L. Traynor, IAS/Park City
Mathematics series, \textbf{7} (1999), pp.\ 143--230.

\bibitem[SZ]{SZ}
D. Salamon, E. Zehnder,
Morse theory for periodic solutions of Hamiltonian systems and the
Maslov index, \emph{Comm.\ Pure Appl.\ Math.}, \textbf{45} (1992),
1303--1360.

\bibitem[Sc]{Sch}
M. Schwarz,
On the action spectrum for closed symplectically aspherical manifolds,
\emph{Pacific J.\ Math.}, \textbf{193} (2000), 419--461.

\bibitem[ST]{ST}
B. Siebert, G. Tian, 
On quantum cohomology rings of Fano manifolds and a formula of Vafa
and Intriligator, \emph{Asian J.\ Math.}, \textbf{1} (1997), 679--695.

\bibitem[Ta]{Ta} 
H. Tamvakis, 
Gromov-Witten invariants and quantum cohomology of
Grassmannians. \emph{Topics in cohomological studies of algebraic
  varieties}, 271--297, Trends Math., Birkh\"auser, Basel, 2005.

\bibitem[Vi1]{Vi}
C. Viterbo, 
Equivariant Morse theory for starshaped Hamiltonian systems,
\emph{Trans.\ Amer.\ Math.\ Soc.}, \textbf{311} (1989), 621--655. 


\bibitem[Vi2]{Vi:gen}
C. Viterbo,
Symplectic topology as the geometry of generating functions,
\emph{Math.\ Ann.}, \textbf{292} (1992), 685--710.

\bibitem[U1]{U1}
M. Usher, 
Spectral numbers in Floer theories, \emph{Compos.\ Math.},
\textbf{144} (2008),  1581--1592.

\bibitem[U2]{U2}
M. Usher, 
Deformed Hamiltonian Floer theory, capacity estimates, and Calabi
quasimorphisms, Preprint 2010, arXiv:1006.5390.

\end{thebibliography}
\end{document}